\documentclass[12pt,leqno]{article}

\usepackage{amsmath,amssymb,amsthm}
\usepackage[all]{xy}
\usepackage{lscape}

\makeatletter

\newtheorem{theorem}{Theorem}[section]

\newtheorem{lemma}[theorem]{Lemma}
\newtheorem{corollary}[theorem]{Corollary}

\theoremstyle{definition}

\theoremstyle{remark}

\theoremstyle{remark}
\newtheorem{remark}[theorem]{Remark}

\def\({{\rm (}}
\def\){{\rm )}}

\let\Mathrm\operator@font
\let\Cal\mathcal
\let\Bbb\mathbb

\def\standop#1{\mathop{\Mathrm #1}\nolimits}
\def\difstop#1#2{\expandafter\def\csname #1\endcsname{\standop{#2}}}
\def\defstop#1{\difstop{#1}{#1}}

\defstop{AB}
\defstop{ann}
\defstop{Ass}

\difstop{charac}{char}
\defstop{CMFI}
\defstop{codim}
\defstop{Coh}
\defstop{coht}
\defstop{Coker}
\defstop{Cone}
\defstop{Cl}
\defstop{Cox}
\defstop{cosk}

\defstop{depth}
\defstop{Div}
\defstop{div}

\defstop{EM}
\defstop{embdim}
\defstop{End}
\defstop{ev}
\defstop{Ext}

\defstop{Flat}
\defstop{Func}
\defstop{Fpqc}

\defstop{Gal}
\def\GL{\text{\sl{GL}}}
\defstop{GD}
\defstop{Good}

\difstop{height}{ht}
\defstop{Hom}

\difstop{Image}{Im}
\defstop{Inv}
\defstop{ind}

\defstop{Ker}

\defstop{Lch}
\defstop{length}
\defstop{Lin}
\defstop{Lqc}
\defstop{lqc}
\defstop{LQ}
\defstop{LM}
\defstop{Lie}

\defstop{Mat}
\defstop{Max}
\defstop{Min}
\defstop{Mod}
\defstop{Mor}
\defstop{MCM}
\defstop{Map}

\defstop{Nerve}
\defstop{NonSFR}
\defstop{NonCMFI}
\defstop{NonNor}
\defstop{Nor}

\defstop{ob}

\defstop{PA}
\defstop{PM}
\defstop{Proj}
\defstop{Prin}
\defstop{Pic}

\defstop{Ps}

\defstop{Qch}
\defstop{qch}

\defstop{rad}
\defstop{rank}
\def\red{_{\Mathrm{red}}}
\defstop{res}
\defstop{Reg}
\defstop{Ref}

\defstop{Sh}

\def\SL{\text{\sl{SL}}}

\defstop{Spec}
\defstop{supp}
\defstop{Supp}
\defstop{Sym}
\defstop{Sing}
\defstop{SFR}
\defstop{Soc}

\defstop{Ob}

\difstop{tdeg}{trans.deg}

\defstop{Tor}
\difstop{trace}{tr}

\defstop{Zar}

\def\bmu{\boldsymbol{\mu}}

\def\O{\Cal O}

\def\sdarrow#1{\downarrow\hbox to 0pt{\scriptsize$#1$\hss}}
\def\suarrow#1{\uparrow\hbox to 0pt{\scriptsize$#1$\hss}}
\def\ssearrow#1{\searrow\hbox to 0pt{\scriptsize$#1$\hss}}

\def\section{\@startsection{section}{1}{\z@ }%
  {-3.5ex plus -1ex minus -.2ex}{2.3ex plus .2ex}{\bf }}

\long\def\refname{\par\kern -3ex
  \begin{center}\rm R\sc{eferences}\end{center}\par\kern 
  -2ex}

\def\@seccntformat#1{\csname the#1\endcsname.\quad}

\def\@@@sect#1#2#3#4#5#6[#7]#8{%
  \ifnum #2>\c@secnumdepth 
  \def \@svsec {}\else \refstepcounter {#1}%
  \def\@svsec{}
  \fi 
  \@tempskipa #5\relax 
  \ifdim \@tempskipa >\z@ 
  \begingroup #6\relax \@hangfrom {\hskip #3\relax 
    \@svsec}{\interlinepenalty \@M #8\par }\endgroup 
  \csname #1mark\endcsname {#7}
  \else 
  \def \@svsechd {#6\hskip #3\@svsec #8\csname #1mark\endcsname {#7}}
  \fi \@xsect {#5}}

\def\@@@startsection#1#2#3#4#5#6{%
  \if@noskipsec \leavevmode \fi \par \@tempskipa #4\relax \@afterindenttrue 
  \ifdim \@tempskipa <\z@ \@tempskipa -\@tempskipa \@afterindentfalse 
  \fi \if@nobreak \everypar {}\else \addpenalty {\@secpenalty }\addvspace 
  {\@tempskipa }\fi \@ifstar {\@ssect {#3}{#4}{#5}{#6}}{\@dblarg 
    {\@@@sect {#1}{#2}{#3}{#4}{#5}{#6}}}}

\def\theparagraph{\thesection.\arabic{paragraph}}
\def\aparagraph{\@@@startsection{paragraph}{2}{\z@ }%
  {1.75ex plus .2ex minus .15ex}{-1em}{\bf(\theparagraph) } }
\def\paragraph{\@@@startsection{paragraph}{2}{\z@ }%
  {1.75ex plus .2ex minus .15ex}{-1em}{}{\bf(\theparagraph)} }

\let\c@theorem\c@paragraph

\title{Classification of the linearly reductive finite subgroup schemes of
$\SL_2$}
\author{M{\sc itsuyasu} H{\sc ashimoto}}
\date{\normalsize
Department of Mathematics, Okayama University\\
Okayama 700--8530, JAPAN\\
E-mail: {\small \tt mh@okayama-u.ac.jp}\\
~\\
Dedicated to Professor Ngo Viet Trung\\
on the occasion of his sixtieth birthday
}

\begin{document}

\maketitle

\begin{abstract}
We classify the linearly reductive finite subgroup schemes $G$ 
of $\SL_2=\SL(V)$ 
over an algebraically closed field $k$ of positive characteristic,
up to conjugation.
As a corollary, we prove that such $G$ is in one-to-one
correspondence with an isomorphism class of 
two-dimensional $F$-rational Gorenstein
complete local rings with the coefficient field $k$
by the correspondence $G\mapsto ((\Sym V)^G)\,\widehat{~}$.
\end{abstract}

\section{Introduction}
The classification of the finite subgroups of $\SL_2(\Bbb C)$ is well-known
(\cite[Section~26]{Dornhoff}, 
\cite[(6.2)]{LW}, see Theorem~\ref{dornhoff.thm}), 
and such a group corresponds to a Dynkin diagram
of type A, D, or E.
A two-dimensional singularity is Gorenstein and rational if and only if
it is a quotient singularity by a finite subgroup of $\SL_2(\Bbb C)$,
and such singularities (also called Kleinian singularities) are classified
via these subgroups, see \cite{Durfee}.
Indeed, a two-dimensional singularity is Gorenstein and rational
if and only if it is a quotient singularity by a finite subgroup of $\SL_2$.

It is known that the $F$-rationality is the characteristic $p$ version of
the rational singularity.
More precisely, a finite-type algebra over a field of characteristic zero
has rational singularities if and only if its modulo $p$ reduction is
$F$-rational for almost all prime numbers $p$ \cite{Smith}, \cite{Hara}.
The two-dimensional complete local $F$-rational Gorenstein rings 
over an algebraically closed field $k$ of characteristic $p>0$ is
classified using Dynkin diagrams A, D, and E, based on
Artin's classification of rational double points 
\cite{Artin}, see \cite{WY}, \cite{HL}.
Then we might well ask whether such a ring is obtained as an 
invariant subring $k[[x,y]]^G$ with $G$ a finite subgroup of $\SL_2=\SL(V)$, 
where $V=kx\oplus ky$.

Before considering this question, we have to consider several things.

First, any finite subgroup of $\SL_2(\Bbb C)$ is small in the sense that
it does not have a pseudo-reflection, where an element $g$ of $\GL(V)$ is
called a pseudo-reflection if $\rank(g-1_V)=1$.
This is important in studying the ring of invariants.
If $G$ is a small finite subgroup of $\GL(V)$ ($V=\Bbb C^2$), then
$G$ can be recovered from $\hat R=\hat S^G$, where $\hat S$ 
is the completion of 
$S=\Sym V$, in the sense that
the fundamental group of $\Spec \hat 
R\setminus\{0\}$ is $G$, where $0$ is the
unique closed point.
Moreover, 
the category of maximal Cohen--Macaulay modules of $\hat R$ is canonically
equivalent to the category of $\hat S$-finite $\hat S$-free $(G,
\hat S)$-modules \cite[(10.9)]{Yoshino}.
However, this is not the case for $\SL_2(k)$ with $\charac(k)>0$.
Indeed, a finite subgroup of $\SL_2(k)$ may have a transvection,
where $g\in\GL(V)$ is called a transvection if it is a pseudo-reflection and
$g-1_V$ is nilpotent.
Even if $G$ is a non-trivial subgroup of $\SL_2$, 
$\hat S^G$ may be a formal power series ring again, 
see \cite[Proposition~4.6]{KS}.

Next, even if $G$ is a finite subgroup of $\SL(V)$, the ring of invariants
$R=(\Sym V)^G$ may not be $F$-regular.
Indeed, Singh \cite{Singh} proved that if $G$ is the alternating group 
$A_n$ acting canonically on $V=k^n$, then $R=(\Sym V)^G$ is
strongly $F$-regular if and only if $p=\charac(k)$ does not divide
the order $(n!)/2$ of $G=A_n$.
More generally, Yasuda \cite{Yasuda} proved that if $G$ is a small subgroup of
$\GL(V)$, then the ring of invariants $(\Sym V)^G$ is strongly $F$-regular
if and only if $p=\charac(k)$ does not divide the order of $G$.

So we want to classify the subgroups $G\subset\SL_2$ with the order of
$G$ is not divisible by $p=\charac(k)$.
It is easy to see that such $G$ must be small.
The classification is known (see Theorem~\ref{dornhoff.thm}), 
and the result is the same as that over $\Bbb C$,
except that small $p$ which divides the order $|G|$ of $G$ is not allowed.
More precisely, for the type $(A_n)$, $p$ must not divide $n+1$, 
for $(D_n)$, $p$ must not divide $4n-8$, and we must have 
$p\geq 5$, $p\geq 5$, $p\geq 7$ for type $(E_6)$, $(E_7)$, and $(E_8)$,
respectively.
However, the restriction on $p$ for the 
classification of two-dimensional $F$-rational
Gorenstein complete local rings is different \cite{HL}, and it is
$p$ arbitrary for $(A_n)$, $p\geq 3$ for $(D_n)$, and
$p\geq 5$, $p\geq 5$, $p\geq 7$ for type $(E_6)$, $(E_7)$, and $(E_8)$,
respectively.

The purpose of this paper is to show the gap occuring on the type $(A_n)$ and
$(D_n)$ comes from the non-reduced group schemes, as shown in
Theorem~\ref{main.thm}.
As a corollary, we show that all the two-dimensional $F$-rational 
Gorenstein complete local rings with the algebraically closed
coefficient field appear as the ring of invariants under the action of
a linearly reductive finite subgroup scheme of $\SL_2$, see
Corollary~\ref{main-cor.thm}.
This is already pointed out by Artin \cite{Artin} for the type $(A_n)$,
and is trivial for $(E_6)$, $(E_7)$, and $(E_8)$ because of the order
of the group and the restriction on $p$.
What is new in this paper is the case $(D_n)$.
At this moment, the author does not know how to recover the group
scheme $G$ from $R=S^G$.
So although the classification of $R$, the two-dimensional $F$-rational
Gorenstein singularities are well-known, the classification of $G$ seems to be
nontrivial for the author.
As a result, we can recover $G$ from $R$ in the sense that the correspondence
from $G$ to $\hat R=\hat S^G$ is one-to-one.

The key to the proof is Sweedler's theorem (Theorem~\ref{sweedler.thm})
which states that a connected linearly reductive group scheme
over a field of positive characteristic is abelian.

The author thanks Professor Kei-ichi Watanabe for valuable advice.

\section{Preliminaries}

\paragraph
Let $k$ be a field.
For a $k$-scheme $X$, we denote the ring $H^0(X,\O_X)$ by $k[X]$.

We say that an affine algebraic $k$-group scheme $G$ is
linearly reductive if any $G$-module is semisimple.

\begin{lemma}\label{short.thm}
Let
\[
1\rightarrow N \rightarrow G\rightarrow H\rightarrow 1
\]
be an exact sequence of affine algebraic $k$-group schemes.
Then $G$ is linearly reductive if and only if $H$ and $N$ are linearly
reductive.
\end{lemma}

\begin{proof}
We prove the \lq if' part.
If $M$ is a $G$-module, then the Lyndon-Hochschild-Serre spectral sequence
\cite[(I.6.6)]{Jantzen}
\[
E_2^{p,q}=H^p(H,H^q(N,M))\Rightarrow H^{p+q}(G,M)
\]
degenerates, and $E_2^{p,q}=0$ for $(p,q)\neq (0,0)$ by assumption.
Thus $H^n(G,M)=0$ for $n>0$, as required.

We prove the \lq only if' part.
First, given a short exact sequence of $H$-modules, it is also a short
exact sequence of $G$-modules by restriction.
By assumption, it $G$-splits, and hence it $H$-splits.
Thus any short exact sequence of $H$-modules $H$-splits, and $H$ is linearly
reductive.
Next, we prove that $N$ is linearly reductive.
Let $M$ be a finite dimensional $N$-module.
Then there is a spectral sequence
\[
E_2^{p,q}=H^p(G,R^q\ind_N^G(M))\Rightarrow H^{p+q}(N,M),
\]
see \cite[(I.4.5)]{Jantzen}.
As $G/N\cong H$ is affine, $R^q\ind_N^G(M)=0$ $(q>0)$ by
\cite[(I.5.13)]{Jantzen}.
As $G$ is linearly reductive by assumption,
$E_2^{p,q}=0$ for $(p,q)\neq (0,0)$.
Thus $H^n(N,M)=0$ for $n>0$, and thus $N$ is linearly reductive.
\end{proof}

\paragraph
Let $C=(C,\Delta,\varepsilon)$ be a $k$-coalgebra.
An element $c\in C$ is said to be group-like if $c\neq 0$ and 
$\Delta(c)=c\otimes c$
\cite{Sweedler}.
If so, $\varepsilon(c)=1$.
The set of group-like elements of $C$ is denoted by $\Cal X(C)$.
Note that $\Cal X(C)$ is linearly independent.

Let $H$ be a $k$-Hopf algebra.
Then for $h\in \Cal X(H)$, $\Cal S(h)=h^{-1}$, where $\Cal S$ is the antipode.
Note that $\Cal X(H)$ is a subgroup of the unit group $H^\times$.
We denote $\GL_1=\Spec k[t,t^{-1}]$ with $t$ group-like 
by $\Bbb G_m$, and its subgroup scheme $\Spec k[t]/(t^r-1)$ by
$\bmu_r$ for $r\geq 0$.
Note that $\bmu_r$ represents the group of the $r$th roots of unity, but
it is not a reduced scheme if $\charac(k)=p$ divides $r$.

\paragraph
In the rest of this paper, let $k$ be algebraically closed.
For an affine algebraic group scheme $G$ over $k$, let $\Cal X(G)$ denote
the group of characters (one-dimensional representations) of $G$.
Note that $\Cal X(G)$ is canonically identified with $\Cal X(k[G])$, 
see \cite[(2.1)]{Waterhouse}.

\begin{lemma}
Let $G$ be an affine algebraic $k$-group scheme.
Then the following are equivalent.
\begin{enumerate}
\item[\bf 1] $G$ is abelian \(that is, the product is commutative\) and
linearly reductive.
\item[\bf 2] $G$ is linearly reductive, and any simple $G$-module
is one-dimensional.
\item[\bf 3] $G$ is diagonalizable.
That is, a closed subgroup scheme of a torus $\Bbb G_m^n$.
\item[\bf 4] The coordinate ring $k[G]$ is group-like as a coalgebra.
That is, $k[G]$ is the group ring $k\Gamma$, where $\Gamma=\Cal X(G)$.
\item[\bf 5] $G$ is a finite direct product of $\Bbb G_m$ and $\bmu_r$ with
$r\geq 2$.
\end{enumerate}
\end{lemma}

\begin{proof}
{\bf 1$\Rightarrow$2}
Follows easily from \cite[(8.0.1)]{Sweedler}.

{\bf 2$\Rightarrow$3}
Take a finite dimensional faithful $G$-module $V$
(this is possible \cite[(3.4)]{Waterhouse}).
Take a basis $v_1,\ldots,v_n$ of $V$ such that each $kv_i$ is a
one-dimensional $G$-submodule of $V$.
Then the embedding $G\rightarrow \GL(V)$ factors through
$\GL(kv_1)\times\cdots\times \GL(kv_n)\cong\Bbb G_m^n$.

{\bf 3$\Rightarrow$4}
Let $G\subset \Bbb G_m^n=T$.
Then $k[T]$ is a Laurent polynomial ring $k[t_1^{\pm1},\ldots,t_n^{\pm 1}]$.
As each Laurent monomial $t_1^{\lambda_1}\cdots t_n^{\lambda_n}$ is group-like,
$k[T]$ is generated by its group-like elements.
This property is obviously inherited by its quotient Hopf algebra
$k[G]$, and we are done.

{\bf 4$\Rightarrow$5} Apply the fundamental theorem of abelian groups on
$\Gamma$.

{\bf 5$\Rightarrow$4$\Rightarrow$2$\Rightarrow$1} is easy.
\end{proof}

\paragraph
The category of finitely generated abelian groups and the
category of diagonalizable $k$-group schemes are contravariantly
equivalent with
the equivalences $\Gamma\mapsto \Spec (k\Gamma)$ and
$G\mapsto \Cal X(G)$.
For a diagonalizable $k$-group scheme $G$, a $G$-module is identified with
an $\Cal X(G)$-graded $k$-vector space.
A $G$-algebra is nothing but a $\Cal X(G)$-graded $k$-algebra.

\paragraph
For a diagonalizable group scheme $G=\Spec k\Gamma$, the closed
subgroup schemes
$H$ of $G$ is in one-to-one correspondence with the quotient groups $M$ of 
$\Gamma$ with the correspondence $H\mapsto \Cal X(H)$ and
$M\mapsto \Spec kM$.
In particular, the only closed 
subgroup schemes of $\Bbb G_m$ is $\bmu_r$ with $r\geq 0$,
since the only quotient groups of $\Bbb Z$ are $\Bbb Z/r\Bbb Z$.

The following is due to Sweedler \cite{Sweedler2}.

\begin{theorem}\label{sweedler.thm}
Let $G$ be a connected linearly reductive affine algebraic $k$-group 
scheme over an algebraically closed field of positive characteristic $p$.
Then $G$ is an abelian group \(and hence is diagonalizable\).
So $G$ is, up to isomorphisms, of the form
\[
\Bbb G_m^r\times \bmu_{p^{e_1}}\times\cdots\times \bmu_{p^{e_s}}
\]
for some $r\geq 0$, $s\geq 0$, and $e_1\geq \cdots e_s\geq 1$.
\end{theorem}

\paragraph
Let $G$ be an affine algebraic $k$-group scheme.
Note that $\Spec k$, $G\red$ and $G\red\times G\red$ are all reduced.
Hence the unit map $e:\Spec k\rightarrow G$, the inverse 
$\iota:G\red\rightarrow G$, and the product $\bmu:G\red\times G\red\rightarrow 
G$ all factor through $G\red\hookrightarrow G$, and so  $G\red$ is a 
closed subgroup scheme of $G$.
Thus $G\red$ is $k$-smooth.

\paragraph
We denote the identity component (the connected component containing the
identity element) of $G$ by $G^\circ$.
As $G\red\hookrightarrow G$ is a homeomorphism and $G\red$ is $k$-smooth,
each connected component of $G$ is irreducible, and is isomorphic to 
$G^\circ$.
As $\Spec k$, $G^\circ$, and $G^\circ\times G^\circ$ are all irreducible,
it is easy to see that
the unit map, the inverse, the product from them all factor through
$G^\circ\hookrightarrow G$, and hence $G^\circ$ is a closed open subgroup of
$G$.
If $C$ is any irreducible component of $G$, then the image of the
map $C\times G^\circ
\rightarrow G$ given by $(g,n)\mapsto gng^{-1}$ is contained in $G^\circ$.
Thus $G^\circ$ is a normal subgroup scheme of $G$.
That is, the map $G\times G^\circ\rightarrow G$ given by 
$(g,n)\mapsto gng^{-1}$ factors through $G^\circ$.

\paragraph As the inclusion $G^\circ \cdot G\red\hookrightarrow G$ is a
surjective open immersion, we have that $G^\circ\cdot G\red=G$.
As $G^\circ$ is an open subscheme of $G$, $G^\circ\cap G\red=G^\circ\red$.
So if $G$ is finite, then $G$ is a semidirect product $G=G^\circ\rtimes G\red$.

\section{The classification}

\paragraph
Throughout this section,
let $k$ be an algebraically closed field of characteristic $p>0$.

The purpose of this section is to classify the linearly reductive finite
subgroup schemes of $\SL_2$ over $k$, up to conjugation.
Our starting point is the reduced case, which is well-known.
Unfortunately, the author does not know the proof of the theorem
below exactly as stated, but the proof in \cite[Section~26]{Dornhoff}
also works for the case of positive characteristic.
See also \cite[Chapter~6, Section~2]{LW}.

\begin{theorem}\label{dornhoff.thm}
Let $k$ be an algebraically closed field of characteristic $p>0$,
and $G$ a finite nontrivial subgroup of $\SL_2$.
Assume that the order $|G|$ of $G$ is not divisible by $p$.
Then $G$ is conjugate to one of the following, where $\zeta_r$ denotes
a primitive $r$th root of unity.
\begin{description}
\item[{$(A_n)$ $(n\geq 1)$}]
The cyclic group generated by
\[
\begin{pmatrix}
\zeta_{n+1} & 0 \\
0 & \zeta_{n+1}^{-1}
\end{pmatrix}.
\]
\item[{$(D_n)$ $(n\geq 4)$}]
The binary dihedral group generated by $(A_{2n-5})$ and
\[
\begin{pmatrix}
0 & \zeta_4 \\
\zeta_4 & 0
\end{pmatrix}.
\]
\item[{$(E_6)$}] The binary tetrahedral group generated by $(D_4)$ and
\[
\frac{1}{\sqrt{2}}
\begin{pmatrix}
\zeta_8^7 & \zeta_8^7 \\
\zeta_8^5 & \zeta_8
\end{pmatrix}
.
\]
\item[{$(E_7)$}] The binary octahedral group generated by $(E_6)$ and $(A_7)$.
\item[{$(E_8)$}] The binary icosahedral group generated by $(A_9)$,
\[
\begin{pmatrix}
0 & 1 \\
-1 & 0
\end{pmatrix},\quad
\text{and}
\quad
\frac{1}{\zeta_5^2-\zeta_5^3}
\begin{pmatrix}
\zeta_5+\zeta_5^{-1} & 1 \\
1 & -(\zeta_5+\zeta_5^{-1})
\end{pmatrix}.
\]
\end{description}
Conversely, if $g=n+1$ \(resp.\ $4n-8$, $24$, $48$, and $120$\) is not
zero in $k$, then $(A_n)$ \(resp.\ $(D_n)$, $(E_6)$, $(E_7)$, and $(E_8)$\) 
above is defined, and is a 
linearly reductive finite subgroup of $\SL_2$ of order $g$.
\end{theorem}

\paragraph
Let $G$ be a linearly reductive finite subgroup scheme of $\SL_2=\SL(V)$.
As the sequence
\[
1\rightarrow G^\circ \rightarrow G\rightarrow G\red\rightarrow 1
\]
is exact, both $G^\circ$ and $G\red$ are linearly reductive
by Lemma~\ref{short.thm}.

\paragraph
First, consider the case that $G$ is abelian.
Then the vector representation $V$ is the direct sum of two 
one-dimensional $G$-modules, say $V_1$ and $V_2$, and hence we may assume that 
$G$ is diagonalized.
As $G\subset\SL_2$, $V_2\cong V_1^*$.
Thus $G\rightarrow \GL(V_1)=\Bbb G_m$ is also a closed immersion, and
$G\cong \bmu_m$ is 
\[
\left\{
\begin{pmatrix}
a & 0 \\
0 & a^{-1}
\end{pmatrix}
\mid a\in \bmu_m
\right\}.
\]

\paragraph
So assume that $G$ is not abelian.
If $G^\circ$ is trivial, then $G=G\red$, and the classification for
this case is done in Theorem~\ref{dornhoff.thm}.
So assume further that $G^\circ$ is non-trivial.

$G^\circ$ is diagonalized
as above, since $G^\circ$ is linearly reductive and connected
(and hence is also abelian by Theorem~\ref{sweedler.thm}).
We have $G^\circ \cong \bmu_r$ with $r=p^e$ for some $e\geq 0$.

\paragraph
We consider the case that $G^\circ$ is contained in the group of
scalar matrices.
In this case, $r=2$ (so $p=2$), as $G\subset \SL_2$.
Then by Maschke's theorem, the order of $G\red$ is odd.
According to the classification in Theorem~\ref{dornhoff.thm}, 
$G\red$ must be of type $(A_n)$ and is cyclic.
This shows that $G$ is abelian, and this is a contradiction.

\paragraph
So $G^\circ$ is not contained in the group of scalar matrices.
Note that if $a,b,c,d\in k$ with $ad-bc=1$ and
\[
\begin{pmatrix}
a & b \\
c & d
\end{pmatrix}
\begin{pmatrix}
\zeta_r & 0 \\
0 & \zeta_r^{-1}
\end{pmatrix}
=
\begin{pmatrix}
\lambda & 0 \\
0 & \mu
\end{pmatrix}
\begin{pmatrix}
a & b \\
c & d
\end{pmatrix}
\]
for some $\lambda,\mu\in A=k[T,T^{-1}]/(T^r-1)$, where 
$\zeta_r$ is the image of $T$ in $A$, then
(1) $\lambda=\zeta_{r}$, $\mu=\zeta_r^{-1}$ and $b=c=0$, or
(2) $\lambda=\zeta_{r}^{-1}$, $\mu=\zeta_r$ and $a=d=0$.
This is because $\zeta_r\neq \zeta_r^{-1}$.
Then it is easy to see that the centralizer $C=Z_G(G^\circ)$ is
contained in the subgroup of diagonal matrices in $\SL_2$.
As we assume that $G$ is not abelian, $C\neq N_G(G^\circ)=G$.
Clearly, $C\red$ has index two in $G\red$.
This shows that the order of $G\red$ is divided by $2$.
By Maschke's theorem, $p\neq 2$.
There exists some matrix
\[
\begin{pmatrix}
0 & b \\
-b^{-1} & 0
\end{pmatrix}
\]
in $G\red$ for some $b\in k^\times$.
After taking conjugate by
\[
\begin{pmatrix}
b^{-1/2}\zeta_8 & 0\\
0 & b^{1/2}\zeta_8^{-1}
\end{pmatrix},
\]
we obtain the group scheme of type $(D_n)$ below 
(see Theorem~\ref{main.thm}) for appropriate $n$.

In conclusion, we have the following.

\begin{theorem}\label{main.thm}
Let $k$ be an algebraically closed field of arbitrary characteristic $p$
\(so $p$ is a prime number, or $\infty$\).
Let $G$ be a linearly reductive finite subgroup scheme of $\SL_2$.
Then, up to conjugation, $G$ agrees with one of the following, 
where $\zeta_r$ denotes
a primitive $r$th root of unity.
\begin{description}
\item[{$(A_n)$ $(n\geq 1)$}]
The group scheme $\bmu_{n+1}$ lying in $\SL_2$ as
\[
\left\{
\begin{pmatrix}
a  & 0 \\
0 & a^{-1}
\end{pmatrix}.
\mid
a\in\bmu_{n+1}
\right\}.
\]
\item[{$(D_n)$ $(n\geq 4)$}]
$p\geq 3$.
The subgroup scheme generated by $(A_{2n-5})$ and
\[
\begin{pmatrix}
0 & \zeta_4 \\
\zeta_4 & 0
\end{pmatrix}.
\]
\item[{$(E_6)$}] $p\geq 5$.
The binary tetrahedral group generated by $(D_4)$ and
\[
\frac{1}{\sqrt{2}}
\begin{pmatrix}
\zeta_8^7 & \zeta_8^7 \\
\zeta_8^5 & \zeta_8
\end{pmatrix}
.
\]
\item[{$(E_7)$}] $p\geq 5$.
The binary octahedral group generated by $(E_6)$ and $(A_7)$.
\item[{$(E_8)$}] $p\geq 7$.
The binary icosahedral group generated by $(A_9)$,
\[
\begin{pmatrix}
0 & 1 \\
-1 & 0
\end{pmatrix},\quad
\text{and}
\quad
\frac{1}{\zeta_5^2-\zeta_5^3}
\begin{pmatrix}
\zeta_5+\zeta_5^{-1} & 1 \\
1 & -(\zeta_5+\zeta_5^{-1})
\end{pmatrix}.
\]
\end{description}
Conversely, any of above is a
linearly reductive finite subgroup scheme of $\SL_2$, and 
a different type gives a non-isomorphic group scheme.
\end{theorem}

\paragraph
For a finite $k$-group scheme $G$ over $k$, we define $|G|:=\dim_k k[G]$.
Then in the theorem, $|G|$ is $n+1$ for $(A_n)$, $4n-8$ for $(D_n)$, and
24, 48, and 120 for $(E_6)$, $(E_7)$, and $(E_8)$, respectively.
This is independent of $p$, and hence is the same as the case for $p=\infty$.

\begin{corollary}\label{main-cor.thm}
Let $k$ be an algebraically closed field of positive characteristic.
Let $\hat R$ be a two-dimensional $F$-rational Gorenstein complete local
ring with the coefficient field $k$.
Then there is a linerly reductive finite subgroup scheme $G$ of 
$\SL_2=\SL(V)$, where $V=k^2$, such that 
the completion of $(\Sym V)^G$ with respect to the irrelevant maximal
ideal is isomorphic to $\hat R$.
Conversely, if $G$ is such a group scheme, then the completion of 
$(\Sym V)^G$ is a two-dimensional $F$-rational Gorenstein complete local
ring with the coefficient field $k$.
\end{corollary}

\begin{proof}
This follows from the theorem and the list in \cite[Example~18]{HL}.

Let $u,v$ be the standard basis of $V=k^2$ and $G$ be as in the list of
the theorem.
Let $S=k[u,v]$ and $R=S^G$.

The case that $G=(A_n)$. 
Then a $G$-algebra is nothing but a $\Cal X(G)=\Bbb Z/(n+1)\Bbb Z$-graded
$k$-algebra.
$S$ is a $G$-algebra with $\deg u=1$ and $\deg v=-1$,
and $R=S_0$, the degree $0$ component with respect to this grading.
Set $x=u^{n+1}$, $y=-v^{n+1}$ and 
$z=uv$.
Then it is easy to see that $R=k[x,y,z]$.
Obviously, it is a quotient of $R_1=k[X,Y,Z]/(XY+Z^{n+1})$.
As $R_1$ is a normal domain of dimension two, $R_1=R$.
So $\hat R$ is of type $(A_n)$.

The case that $G=(D_n)$.
Set
$x=uv(u^{2n-4}-(-1)^{n-2}v^{2n-4})$,
$y=-2^{2/(n-1)}u^2v^2$ and
$z=2^{-1/(n-1)}(u^{2n-4}+(-1)^{n-2}v^{2n-4})$.
Then
\[
k[x,y,z]\subset R=S^G=(S^{G'})^{G/G'}
\subset k[u^{2n-4},uv,v^{2n-4}]=S^{G'}\subset S,
\]
where $G'$ is the group scheme of type $(A_{2n-5})$.
Note that $k[x,y,z]$ is a quotient of
$R_1=k[x,y,z]=k[X,Y,Z]/(X^2+YZ^2+Y^{n-1})$.
As $R_1$ is a two-dimensional normal domain, $R_1\to k[x,y,z]$ is
an isomorphism, and hence $k[x,y,z]$ is normal.
It is easy to see that $Q(S^{G'})=k(x,y,z,uv)$ and $[k(x,y,z,uv):k(x,y,z)]
\leq 2$.
As $|G/G'|=2$, $R_1=k[x,y,z]\rightarrow R$ is finite and birational.
As $R_1$ is normal, $R_1=R$.
Thus $\hat R$ is of type $(D_n)$.

The cases of constant groups $G=(E_6),(E_7),(E_8)$ are well-known \cite{LW},
and we omit the proof.
\end{proof}

\begin{remark}
Note that the converse in the corollary is also checked theoretically.
As $G$ is linearly reductive, $R=(\Sym V)^G$ is a direct summand subring of
$S=\Sym V$, and hence is strongly $F$-regular.
Thus its completion is also strongly $F$-regular, see for example, 
\cite[(3.28)]{Hashimoto}.
Gorenstein property of $R$ is a consequence of
\cite[(32.4)]{ETI}.

Nevertheless, at this moment, the author does not know a theoretical reason
why $G$ can be recovered from 
the isomorphism class of $\hat R$
(this is true, as can be seen from the result of the classification).
\end{remark}

\begin{remark}
Let $V$ and $G$ be as above.
Set $S:=(\Sym V)^G$, and let $\hat S$ be its completion with respect 
to the irrelevant maximal ideal so that $\hat S\cong k[[x,y]]$.
As $G^\circ$ is infinitesimal, $\hat S^{G^\circ}\rightarrow \hat S$ is
purely inseparable.
So $\Spec \hat S^{G^\circ}\setminus 0$ is simply connected.
As $\Spec \hat S^{G^\circ}\setminus 0 \rightarrow \Spec \hat S^{G}\setminus 0$
is a Galois covering of the Galois group $G/G^\circ=G\red$, the
fundamental group of $\Spec \hat S^G\setminus 0$ is $G\red$, which is linearly
reductive, as stated in \cite{Artin}.
\end{remark}

\end{document}